\documentclass[11pt]{article}
\usepackage{amsthm, amssymb, geometry, mathrsfs}
\usepackage[T1]{fontenc}
\usepackage[all]{xy}

\usepackage{graphicx, amsmath, sseq}
\begin{document}
\title{On finite derived quotients of 3-manifold groups}
\author{Will Cavendish\footnote{Supported in part by NSF} }

\theoremstyle{theorem}
\newtheorem{theorem}{Theorem}[section]
\newtheorem{thm}{Theorem}
\newtheorem{prop}{Proposition}
\newtheorem{lemma}{Lemma}
\newtheorem{cor}{Corollary}
\newtheorem{conj}{Conjecture}
\newtheorem{Def}{Definition}
\newtheorem{Quest}{Question}
\newcommand{\dc}{\textrm{d}_\mathcal{C}}
\newcommand{\fpg}{\mathcal{CP}_{g,0}}
\newcommand{\fpgn}{\mathcal{CP}_{g,n}}
\newcommand{\diam}{\mathrm diam}
\newcommand{\sys}{\mathrm{sys}\,}
\newcommand{\PP}{P}
\newcommand{\BB}{{\mathcal B}}
\newcommand{\R}{{\mathbb R}}
\newcommand{\MM}{{\mathcal M}}
\newcommand{\area}{{\rm Area}}
\newcommand{\Sym}{{\rm Sym}}
\newcommand{\Stab}{{\rm Stab}}
\newcommand{\Ext}{{\rm Ext}}
\newcommand{\Hom}{{\rm Hom}}
\newcommand{\Tor}{{\rm Tor}}
\newcommand{\coker}{{\rm coker}}
\newcommand{\into}{{\hookrightarrow}}
\newcommand{\ord}{{\rm ord}}
\renewcommand{\thefootnote}{\fnsymbol{\dagger}}

\author{Will Cavendish\footnote{Supported in part by NSF Grant DMS-1204336} }
\date{}
\maketitle

\begin{abstract} This paper studies the set of finite groups appearing as $\pi_1(M)/\pi_1(M)^{(n)}$, where $M$ is a closed, orientable 3-manifold and $\pi_1(M)^{(n)}$ denotes the $n$-th term of the derived series of $\pi_1(M)$.  Our main result is that if $M$ is a closed, orientable 3-manifold, $n\ge 2$, and $G\cong \pi_1(M)/\pi_1(M)^{(n)}$ is finite, then the cup product pairing $H^2(G)\otimes H^2(G)\to H^4(G)$ has cyclic image $C$, and the pairing $H^2(G)\otimes H^2(G)\stackrel{\smile}{\longrightarrow} C$ is isomorphic to the linking pairing $H_1(M)_{\textrm{Tors}}\otimes H_1(M)_{\textrm{Tors}}\to \mathbb{Q}/\mathbb{Z}$.
\end{abstract}

\section{Introduction}

One of the most elementary invariants of a connected topological space $M$ is its first homology group $H_1(M)$, which, via Hurewitz's theorem may be expressed as $\pi_1(M)/\pi_1(M)^{(1)}$, the quotient of the fundamental group of $M$ by its derived subgroup.  Somewhat more mysterious are the topological invariants given by the higher derived quotients of the fundamental group, $\pi_1(M)/\pi_1(M)^{(n)}$, where $\pi_1(M)^{(n)}$ denotes the $n$-th term of the derived series of $\pi_1(M)$.  This paper studies the groups that appear as $\pi_1(M)/\pi_1(M)^{(n)}$ when $M$ is a closed, orientable 3-manifold, in the special case that $\pi_1(M)/\pi_1(M)^{(n)}$ is finite.  Our interest in these groups is motivated by the following well-known question:

\begin{Quest}  Let $M$ be a closed, orientable $3$-manifold.  If $[\pi_1(M):\pi_1(M)^{(n)}]$ is finite for all $n$, does the derived series of $\pi_1(M)$ stabilize, i.e. is $\pi_1(M)^{(i)}$ a perfect group for some $i$?
\end{Quest}

When the derived series of $\pi_1(M)$ stabilizes, it is well-known that $\pi_1(M)/\pi_1(M)^{(i)}$ is a solvable group with 4-periodic cohomology.  Groups of this form are extremely rare, and are isomorphic to the product of a trivial, dihedral, or generalized quaternion group with a cyclic group of relatively prime order (see \cite{Mil}).  The main thrust of this paper is that the finite groups that appear as $\pi_1(M)/\pi_1(M)^{(n)}$ satisfy restrictive group-theoretic constraints as well.  

These constraints do not appear when $n=1$, since by taking connected sums of lens spaces one can easily show that any finite abelian group appears as $\pi_1(M)/\pi_1(M)^{(1)}\cong H_1(M)$ for a closed orientable 3-manifold $M$.  The question of which finite groups appear as $\pi_1(M)/\pi_1(M)^{(2)}$, however, is already more interesting.  Note that if $\pi_1(M)/\pi_1(M)^{(2)}$ is finite, then $\pi_1(M)^{(1)}/\pi_1(M)^{(2)}$ is a finite group and therefore the maximal abelian cover of $M$ has trivial first Betti number.  Reznikov showed in \cite{Rez} that the fundamental groups of 3-manifolds with this property satisfy a number of non-trivial constraints, and consequently not every metabelian group can appear as $\pi_1(M)/\pi_1(M)^{(2)}$.   The restrictions Reznikov discovered are especially interesting in light of the main result of Cooper and Long in \cite{CL}, which shows that any finite group appears as the group of deck transformations of a regular covering $\rho: M'\to M$ of closed 3-manifolds where $b_1(M')=0$.  This shows that Reznikov's constraints do not arise from obstructions to fixed-point free group actions on rational homology 3-spheres.  

In this paper we build on the themes explored by Reznikov in \cite{Rez} by showing that the cohomology ring of a finite group of the form $\pi_1(M)/\pi_1(M)^{(n)}$ for $n\ge 2$ directly reflects information about the \emph{linking pairing} on $H_1(M)$, which is a non-degenerate bilinear form $H_1(M)_{\textrm{Tors}}\otimes H_1(M)_{\textrm{Tors}}\to \mathbb{Q}/\mathbb{Z}$ whose definition we now recall.  Given an element $[a]\in H_1(M)_{\textrm{Tors}}$ and a loop $\gamma_a$ representing $[a]$, there exists an integer $n\in \mathbb{N}$ such that $n\cdot [a]=0$.  Since the 1-cycle $n\cdot \gamma_a$ is homologically trivial, there exists an immersed oriented surface $\Sigma_a$ in $M$ such that the oriented boundary of $\Sigma_a$ is equal to $n\cdot \gamma_a$.  Given another class $[b]\in H_1(M)_{\textrm{Tors}}$, there exists a loop $\gamma_b$ representing $b$ such that $\gamma_b$ is transverse to $\Sigma_a$.  The value of the linking pairing $\langle [a], [b]\rangle$ is defined by $(\gamma_b \pitchfork \Sigma_a)/n  \in \mathbb{Q}/\mathbb{Z}$, the algebraic intersection number of $\gamma_b$ and $\Sigma_a$ divided by $n$.  

The following theorem shows that the linking pairing on $H_1(M)_{\textrm{Tors}}$ is isomorphic the 2-dimensional cup-product pairing in $H^\ast(G)$ for any finite quotient of $q:\pi_1(M)\to G$ such that $\ker(q)\subseteq \pi_1(M)^{(2)}$:

\begin{thm}\label{mainthm}  Let $M$ be a closed, orientable 3-manifold, let $\Gamma\cong \pi_1(M)$, and let $q: \Gamma \to G$ be a surjective homomorphism such that $\ker(q)\subseteq \Gamma^{(2)}$.  If $G$ is finite, then the cup-product pairing $H^2(G)\otimes H^2(G)\stackrel{}{\rightarrow} H^4(G)$ is non-degenerate and has cyclic image $C<H^4(G)$.  Furthermore, there exists an embedding $i:C\to \mathbb{Q}/\mathbb{Z}$ such that for any $\omega_1, \omega_2\in H^2(G)$,
\[
i(\omega_1\smile \omega_2)=\langle [M]\frown \tilde{q}^\ast(\omega_1), [M]\frown \tilde{q}^\ast(\omega_2) \rangle,  
\]
where $[M]\in H_3(M)$ denotes the fundamental class of $M$, $\langle-,-\rangle$ denotes the linking pairing on $H_1(M)_{\textrm{Tors}}$, and $\tilde{q}: M\to BG$ is a continuous map from $M$ to the classifying space of $G$ such that $\tilde{q}_\ast:\pi_1(M)\to \pi_1(BG)\cong G$ is equal to $q$.
\end{thm}

We remark that Theorem \ref{mainthm} does not hold when the quotient group $\Gamma/\Gamma^{(2)}$ is infinite, even when $H_1(M)=\Gamma/\Gamma^{(1)}$ is finite.  To see this, let $M$ be homeomorphic to $\mathbb{RP}^3\#\mathbb{RP}^3$, the connected sum of two copies of real projective 3-space, and let $G=\pi_1(M)/\pi_1(M)^{(2)}$.  The fundamental group $\pi_1(M)$ is isomorphic to the infinite dihedral group $D_\infty\cong \mathbb{Z}/2\ast \mathbb{Z}/2$, and hence $H_1(M)\cong \mathbb{Z}/2\oplus\mathbb{Z}/2$.  The commutator subgroup of $\mathbb{Z}/2\ast\mathbb{Z}/2$ is isomorphic to $\mathbb{Z}$, so $\pi_1(M)^{(2)}$ is trivial and therefore $G=\pi_1(M)\cong \pi_1(M)/\pi_1(M)^{(2)}\cong \mathbb{Z}/2\ast\mathbb{Z}/2$.  Recall that the infinite dimensional real projective space $\mathbb{RP}^\infty$ is a classifying space for $\mathbb{Z}/2$, and that the cohomology ring $H^\ast(\mathbb{RP}^\infty)$ is generated by a single degree 2 element of order 2.  It follows that the wedge sum $\mathbb{RP}^\infty\vee \mathbb{RP}^\infty$ is a classifying space for $\mathbb{Z}/2\ast\mathbb{Z}/2$.  Since the cohomology ring of a wedge sum of connected spaces is isomorphic to the direct sum of the cohomology rings of the summands modulo the identification of the zeroth cohomology groups, $H^\ast(G)\cong H^\ast(\mathbb{RP}^\infty\vee \mathbb{RP}^\infty)\cong \mathbb{Z}[x,y]/(xy,2x,2y)$, where $\deg(x)=\deg(y)=2$.  The elements $x^2$ and $y^2$ are linearly independent in $H^4(G)$, so the image of the cup-product pairing $H^2(G)\otimes H^2(G)\to H^4(G)$ is not cyclic.  This example can be modified to produce aspherical (indeed hyperbolic) examples of such 3-manifolds using the techniques of Baker, Boileau and Wang in \cite{BBW}. 

As a sample application of Theorem \ref{mainthm}, we demonstrate how it can be used to derive the following result of Reznikov (Theorem 12.5 in \cite{Rez}) from well-known results about $2$-groups and their cohomology rings.

\begin{thm}[Reznikov]  Let $\Gamma$ be the fundamental group of a closed, orientable 3-manifold $M$ such that $H_1(M)\cong \mathbb{Z}/2\oplus\mathbb{Z}/2$.  If $\Gamma/\Gamma^{(2)}$ is finite and $\langle c,c \rangle=0$ for all non-trivial $c\in H_1(M)$, then the Sylow-2 subgroup of $\Gamma/\Gamma^{(2)}$ is equal to the quaternion group of order 8.  If $\Gamma/\Gamma^{(2)}$ is finite and $\langle c,c \rangle $ is non-trivial for some non-trivial $c\in H_1(M)$, then the Sylow-2 subgroup of $\Gamma/\Gamma^{(2)}$ is isomorphic to a generalized quaternion group $Q_{2^k}$ for $k>3$.
\end{thm}

\begin{proof}  Let $G$ denote $\Gamma/\Gamma^{(2)}$, let $S$ denote the Sylow-2 subgroup of $G$, and let $K$ denote $G^{(1)}$.  Note that since $G$ is metabelian, it follows that $K$ is abelian.  Let $K_{(p)}$ denote the $p$-part of $K$ and let $K'$ be the complementary subgroup of $K$ so that $K'\oplus K_{(p)}$.  Note that $K'$ is a characteristic subgroup of $G$, and is therefore normal, and that the order of $G/K'$ is equal to the order of $S$.  Since $S$ is a Sylow subgroup $G/K'\cong S$, and the inclusion $S\into G$ therefore has a right inverse $r:G\to S$.  This shows that $r^\ast: H^\ast(S)\to H^\ast(G)$ is injective.  It is a simple consequence of the universal coefficients theorem (see Lemma \ref{H2surj} in Section \ref{lpandcups} below) that $r^\ast: H^2(S)\to H^2(G)$ is an isomorphism, and it follows from naturality of the cup product that the cup product pairing on $H^2(S)$ is isomorphic to the cup product pairing on $H^2(G)$.  Applying Theorem \ref{mainthm}, the cup product pairing on $H^2(S)$ is therefore isomorphic to the linking form on $H_1(M)_{\textrm{Tors}}$.

We now examine the possibilities for the group $S$.  If $S$ is abelian, then $S\cong \mathbb{Z}/2\oplus\mathbb{Z}/2$ and the cup-product pairing $H^2(S)\otimes H^2(S)\to H^4(S)$ has 3-dimensional image by the K\"unneth theorem.  We may therefore assume that $S$ is non-abelian.   Since $H_1(S)\cong\mathbb{Z}/2\oplus \mathbb{Z}/2$, $S$ is a $2$-group of maximal class (see \cite{Gor} section 5.4), and is therefore isomorphic to either a dihedral group, a semi-dihedral group, or a quaternion group.  The cup-product pairing on the second cohomology of a dihedral group of order $4k$ has image with rank larger than one (see \cite{Han}), and is degenerate on any semi-dihedral group see (\cite{EP}).  It follows that $S$ is isomorphic to $Q_{2^k}$, a generalized quaternion group of order $2^k$.  

The cohomology ring $H^\ast(Q_8)$ of the quaternion group of order 8 has the feature that any $\alpha \in H^2(Q_8)$ satisfies $\alpha^2=0$  (see \cite{Ati}) , whereas the cohomology ring of $H^\ast(Q_{2^k})$ for $k>3$ has 2-dimensional elements with non-trivial squares (see \cite{HS}).  Since the cup product pairing on $H^2(S)$ is isomorphic to the linking pairing on $H_1(M)_{\textrm{Tors}}$, it follows that if $\langle c,c \rangle =0$ for all $c\in H_1(M)_{\textrm{Tors}}$ then $S\cong Q_8$, otherwise $S\cong Q_{2^k}$ for some $k>3$.

\end{proof}

Given a manifold $M$, let $\widetilde{M}_{ab}$ denote the maximal abelian cover of $M$.  It is interesting to note that there are only two isomorphism types of non-degenerate pairings $(\mathbb{Z}/2)^2\otimes (\mathbb{Z}/2)^2\to \mathbb{Q}/\mathbb{Z}$, and that both types appear as pairings on $H_1(M)_{\textrm{Tors}}$ for a closed orientable 3-manifold $M$ such that $b_1(\widetilde{M}_{ab})=0$.  Examples of such manifolds are given by the spaces $S^3/Q_8$ and $S^3/Q_{16}$, where $S^3$ is viewed as the group of unit quaternions and $Q_{2n}$ is realized as the subgroup of $S^3$ generated by $e^{i\pi/n}$ and $j$.  It has been shown by Kawauchi and Kojima in \cite{KK} that every non-degenerate pairing $A\otimes A\to \mathbb{Q}/\mathbb{Z}$ on a finite abelian group $A$ appears as the linking form of a 3-manifold with $b_1(M)=0$.  Given this result, it is interesting to ask the following:

\begin{Quest}  Given a finite abelian group $A$, does every non-degenerate bilinear pairings \linebreak $A\otimes A\to \mathbb{Q}/\mathbb{Z}$ appear as the linking pairing of a 3-manifold $M$ such that $H_1(M)=A$ and $b_1(\widetilde{M}_{ab})=0$? 
\end{Quest}

Note that by Theorem \ref{mainthm}, any pairing that appears in this way also appears as the cup product pairing on $H^2(G)$ for the finite metabelian group $G\cong \pi_1(M)/\pi_1(M)^{(2)}$.

The proof of Theorem \ref{mainthm} breaks into two parts.  The first part, carried out in section \ref{lpandcups}, consists of showing that the linking pairing on $H_1(M)$ can be factored through the cup-product pairing  $H^2(G)\otimes H^2(G)\to H^4(G)$.  One consequence of this factorization is the following theorem, which applies to quotients of $\pi_1(M)$ whose abelianization has maximal order.

\begin{thm}\label{features}  Let $\Gamma$ be the fundamental group of a closed, orientable 3-manifold $M$, let $[M]\in H_3(M)$ denote the fundamental class of $M$, and let $q:\Gamma\to G$ be a surjective homomorphism onto a finite group $G$ such that $\ker(q)\subseteq \Gamma^{(1)}$.  Then 
\begin{itemize}
\item[(i)]$H^2(G)\otimes H^2(G)\to H^4(G)$ is non-degenerate, and
\item [(ii)]$\ord(\tilde{q}_\ast([M]))\cdot H_1(M)=0$, where $\tilde{q}:M\to BG$ is a continuous map such that $\tilde{q}_\ast:\pi_1(M)\to \pi_1(BG)\cong G$ is equal to $q$.
\end{itemize}
\end{thm}

Note that this theorem provides information about how the fundamental class of the manifold $M$ behaves under finite quotient maps.  Indeed, in the special case that $G\cong H_1(M)$, this theorem shows that the image of the homomorphism $\tilde{q}_\ast: H_3(M)\to H_3(G)$ has maximal order, since the annihilator of $H_3(G)$ is equal to the annihilator of $H_1(G)$ when $G$ is an abelian group.

The second part of the proof of Theorem \ref{mainthm}, carried out in section \ref{cyclicity}, establishes the following result using an argument known as spectral sequence comparison: 

\begin{lemma}\label{cupproducts} Let $\Gamma$ be the fundamental group of a closed, orientable 3-manifold $M$, and let $\rho:\pi_1(M)\to G$ be a surjective homomorphism. If $\ker(\rho)\subseteq \Gamma^{(2)}$, then the image of the cup-product pairing $H^2(G)\otimes H^2(G)\to H^4(G)$ is cyclic.
\end{lemma}

As we will show in section 4, Theorem \ref{mainthm} follows easily from Lemma \ref{cupproducts} together with the relationship between the cup-product pairing and the linking pairing established in section \ref{lpandcups}.

\section{The linking pairing on $H_1(M)_{\textrm{Tors}}$ and cup products in quotients of $\pi_1(M)$ with maximal abelianization}\label{lpandcups}

Throughout this paper, we will let $M$ be a closed orientable 3-manifold, we will let $\Gamma$ denote $\pi_1(M)$, $\langle-,-\rangle$ denote the linking pairing on $H_1(M)_{\textrm{Tors}}$, and $[M]\in H_3(M)$ denote the fundamental class of $M$.  For functoriality reasons, we will regard the image of the linking pairing $\langle-,- \rangle$ as an element of $H_0(M,\mathbb{Q}/\mathbb{Z})$, rather than as an element of the abstract group $\mathbb{Q}/\mathbb{Z}$.  It will also be convenient to work the dual pairing $\lambda: H^2(M)_{\textrm{Tors}}\otimes H^2(M)_{\textrm{Tors}}\to H_0(M,\mathbb{Q}/\mathbb{Z})$, defined by $\lambda(\omega_1, \omega_2)= \langle[M]\frown \omega_1, [M]\frown \omega_2 \rangle$.  This pairing satisfies the following well-known identity:
\begin{equation}\label{lambda}
\lambda(\omega_1,\omega_2)=[M]\frown \left(\omega_1\smile \beta^{-1}(\omega_2)\right),
\end{equation}
where $\beta: H^1(M,\mathbb{Q}/\mathbb{Z})\to H^2(M,\mathbb{Z})$ denotes the Bockstein homomorphism arising from the short-exact sequence $0\to \mathbb{Z}\to \mathbb{Q}\to \mathbb{Q}/\mathbb{Z}\to 0$.

The following lemma shows how the linking pairing on $H_1(M)_{\textrm{Tors}}$ relates to the cup product pairing $H^2(G)\otimes H^2(G)\to H^4(G)$.

\begin{lemma}\label{fundid} Let $q: \Gamma\to G$ be a surjective homomorphism onto a finite group $G$, and let $\tilde{q}: M\to BG$ be a continuous map from $M$ to the classifying space of $G$ such that $\tilde{q}_\ast:\pi_1(M)\to \pi_1(BG)\cong G$ equals $q$.  Given $\omega_1, \omega_2\in H^2(G)$, 
\[
\tilde{q}_\ast\left(\lambda( \tilde{q}^\ast(\omega_1), \tilde{q}^\ast(\omega_2))\right)= \tilde{q}_\ast([M])\frown \beta^{-1}(\omega_1\smile \omega_2). 
\]
\end{lemma}

\begin{proof}  We begin by noting that the expressions on the left-hand side of the above identity are well-defined, since $H^2(G)$ is a finite groups and therefore $\tilde{q}^\ast(\omega)\in H^2(M)_{\textrm{Tors}}$.  Note also that since $G$ is finite $H^i(G,\mathbb{Q})=0$ for all $i>0$.  It follows that $\beta: H^1(G,\mathbb{Q}/\mathbb{Z})\to H^2(G,\mathbb{Z})$ is an isomorphism and therefore the map $\beta^{-1}$ appearing on the right-hand side of the above equation is well-defined as well.

Recall that the cup product pairing $H^i(G,A)\otimes H^j(G,\mathbb{Z})\to H^i(G,A\otimes \mathbb{Z})\cong H^i(G,A)$ equips $H^\ast(G,A)$ with the structure of a right $H^\ast(G)$-module.  Given a short exact sequence $0\to A\to B\to C\to 0$ of $G$-modules, the connecting homomorphisms in the long exact sequence
\[
\dots \to H^i(G, A)\to H^i(G,B)\to H^i(G,C)\stackrel{\delta}{\to} H^{i+1}(G,A)\to \dots
\] 
fit together to give an $H^\ast(G)$-module homomorphism $\delta: H^\ast(G,C)\to H^\ast(G,A)$, i.e. given $\alpha\in H^\ast(G,C)$ and $\omega\in H^\ast(G)$, $\delta(\alpha )\smile \omega =\delta(\alpha\smile \omega)$ (see \cite{Bro} Chapter V.3).  Since the Bockstein homomorphism $\beta: H^\ast(G,\mathbb{Q}/\mathbb{Z})\to H^\ast(G,\mathbb{Z})$ is given by the connecting homomorphism in the long exact sequence in cohomology arising from the short exact sequence $0\to \mathbb{Z} \to \mathbb{Q}\to \mathbb{Q}/\mathbb{Z}\to 0$, given $\omega\in H^1(G,\mathbb{Q}/\mathbb{Z})$ and $\eta\in H^1(G,\mathbb{Z})$, $\beta(\omega)\smile \eta= \beta(\omega\smile \eta)$.  It follows that given $\eta_1, \eta_2\in H^2(G,\mathbb{Z})$, $\eta_1\smile \eta_2=\beta(\beta^{-1}(\eta_1)\smile \eta_2)$, and therefore
\begin{equation}\label{cupeqn}
\beta^{-1}(\eta_1\smile \eta_2)= \beta^{-1}(\eta_1)\smile \eta_2.
\end{equation}
Recall that for a continuous map $f:X\to Y$ between topological spaces, the cap product satisfies the following naturality property for $c\in H_i(X), \eta\in H^j(Y)$:
\begin{equation}\label{capid}
f_\ast(c \frown f^\ast(\eta))=f_\ast(c)\frown \eta. 
\end{equation}
Applying these identities together with the identity (\ref{lambda}) for the pairing $\lambda$ and naturality of the cup product and the Bockstein homomorphism, we obtain
\[
\tilde{q}_\ast( \lambda \left(\tilde{q}^\ast(\omega_1), \tilde{q}^\ast(\omega_2)\right))
\overset{(\ref{lambda})}{=}
\tilde{q}_\ast\left([M]\frown (\tilde{q}^\ast(\omega_1)\smile \beta^{-1}(\tilde{q}^\ast(\omega_2)))\right)~~~~~~~
\]
\[
=\tilde{q}_\ast\left([M]\frown (\tilde{q}^\ast(\omega_1)\smile \tilde{q}^\ast(\beta^{-1}(\omega_2)))\right)
=\tilde{q}_\ast\left([M]\frown \tilde{q}^\ast(\omega_1 \smile \beta^{-1}(\omega_2))\right)
\]
\[
\overset{(\ref{capid})}{=}\tilde{q}_\ast([M])\frown (\omega_1 \smile \beta^{-1}(\omega_2))
\overset{(\ref{cupeqn})}{=} \tilde{q}_\ast([M])\frown \beta^{-1}(\omega_1\smile \omega_2).~~~~~~~~~~~~~~~~
\]

\end{proof}

We now turn to the proof of Theorem \ref{features}, which will require several preliminary lemmas.

\begin{lemma}\label{H2surj}  Let $f:X\to Y$ be a continuous map between topological spaces such that $H_2(X)$ and $H_2(Y)$ are finite.  If $f_\ast: H_1(X)\to H_1(Y)$ is an isomorphism, then $f^\ast: H^2(Y)\to H^2(X)$ is an isomorphism.
\end{lemma}

\begin{proof}
By naturality of the universal coefficients exact sequence, we have the following commutative diagram of exact sequences:
 \[
\begin{array}[c]{ccccccccc}
0\to \Ext(H_1(Y),\mathbb{Z}) \to H^2(Y) \to \Hom(H_2(Y),\mathbb{Z}) \to 0\\
~~~~~~~~\downarrow\mbox{\scriptsize $f^\ast$ }~~~~~~~~~~~~ \downarrow  \mbox{\scriptsize $f^\ast$}~~~~~~~~~~~~~~\downarrow\mbox{\scriptsize $f^\ast$ } \\
0\to \Ext(H_1(X),\mathbb{Z}) \to H^2(X) \to \Hom(H_2(X),\mathbb{Z}) \to 0\\.
\end{array}
\]
Since $H_2(Y)$ and $H_2(X)$ are finite, the $\Hom$ terms in the above diagram are trivial and therefore $f^\ast: H^2(Y)\to H^2(X)$ is completely determined by $f^\ast: \Ext(H_1(Y),\mathbb{Z})\to \Ext(H_1(X),\mathbb{Z})$.  If $f_\ast: H_1(Y)\to H_1(X)$ is an isomorphism, then $f^\ast: \Ext(H_1(Y), \mathbb{Z})\to \Ext(H_1(X),\mathbb{Z})$ is an isomorphism by functoriality, so the lemma follows.

\end{proof}

Note that since finite groups have finite second homology groups, Lemma \ref{H2surj} can be applied to any homomorphism between finite groups that induces an isomorphism on the level of abelianizations.  We will use the following simple consequence of this lemma several times in what follows.

\begin{lemma}\label{3mfldversion}  Let $M$ be a rational homology 3-sphere,  let $q: \pi_1(M)\to G$ be a surjective homomorphism onto a finite group, and let $\tilde{q}: M\to BG$ be a continuous homomorphism from $M$ to the classifying space of $G$ such that $\tilde{q}_\ast:\pi_1(M)\to \pi_1(BG)\cong G$ is equal to $q$.  If $\ker(q)\subseteq \pi_1(M)^{(1)}$, then $\tilde{q}^\ast: H^2(G)\to H^2(M)$ is an isomorphism.  
\end{lemma}

\begin{proof}  We claim that the hypotheses of Lemma \ref{H2surj} hold in this setting.  To see that $q_\ast:H_1(M)\to H_1(G)$ is an isomorphism, note that since $\ker(q)\subseteq \Gamma^{(1)}$, the abelianization map $\Gamma\to \Gamma/\Gamma^{(1)}\cong H_1(M)$ factors through $q$.  It follows that the homomorphism $\tilde{q}_\ast:H_1(M)\to H_1(G)$ is injective.  Since the map $q$ is surjective, the induced map is $\tilde{q}_\ast: H_1(M)\to H_1(G)$ is surjective as well.

Since $H_2(G)$ is finite for any finite group $G$, it remains to check that $H_2(M)$ is finite.  This is a simple consequence of Poincar\'e duality and the universal coefficients theorem, since $H_2(M)\cong H^1(M)\cong \Hom(H_1(M),\mathbb{Z})$, and $\Hom(H_1(M),\mathbb{Z})=0$ since $H_1(M)$ is a torsion group.

\end{proof}

 The next lemma we will need is the following well-known result on the values taken by the linking form.

\begin{lemma}\label{order}  Let $M$ be a 3-manifold.  Given $a \in H^2(M)_{\textrm{Tors}}$, there exists $b \in H^2(M)_{\textrm{Tors}}$ such that $\ord(\lambda(\omega,\eta))=\ord(\omega)$.
\end{lemma}

\begin{proof}  Let $A$ denote $H^2(M)_{\textrm{Tors}}$.  It is a well-known consequence of Poincar\'e duality that $\lambda :A\times A\to \mathbb{Q}/\mathbb{Z}$ is non-degenerate, so the homomorphism $A\to \Hom(A,\mathbb{Q}/\mathbb{Z})$ given by $a\mapsto \lambda(a,-)$ is injective.  Given $b\in A$, let $\mu_b=\ord(a)/\ord(\lambda (a,b))$.  Let $d$ be the greatest commont divisor of $\{\mu_b~|~b\in A\}$.  Then 
\[
\frac{\ord(a)}{d}\cdot \lambda(a,b)= \frac{\mu_b\cdot \ord(\lambda(a,b))}{d}\cdot \lambda (a,b)= \frac{\mu_b}{d}\cdot \ord(\lambda(a,b))\cdot \lambda(a,b)=0
\]
for all $b$, so $(\ord(a)/d)\cdot \lambda(a,-)=0$ and therefore  $\ord(\lambda( a,-))$ divides $\ord(a)/d$. Since $a\mapsto \lambda(a,-)$ is an isomorphism, $\ord(\lambda(a,-))=\ord(a)$ so $d=\pm 1$.  It follows that for each prime $p$ dividing $\ord(a)$, there exists an element $b_p$ such that $\mu_{b_p}$ is coprime to $p$ and hence the $p$-part of $\ord(a)$ divides $\ord(\lambda(a,b))$.  By taking a multiple of $b_p$ if necessary, we can assume that $\ord(\lambda(a,b_p))$ is exactly equal to the $p$-part of $\ord(a)$.  Since the sum of a set of elements with pairwise coprime orders $n_1,n_2, \dots , n_\ell$, in an abelian group has order given by $n_1\cdot n_2\cdot \dots \cdot n_\ell$
\[
\ord\left(\lambda(a, \sum_p b_p)\right)=\ord\left(\sum_p \lambda(a, b_p)\right)= \prod_p \ord~\lambda (a, b_p)=\ord(a).
\]    

\end{proof}

The proof of Theorem \ref{features} follows easily from the above lemmas:

\begin{proof}[Proof of Theorem \ref{features}]  Let $q:\Gamma\to G$ be a surjective homomorphism onto a finite group such that $\ker(q)\subseteq \Gamma^{(1)}$.  Note that  By Lemma \ref{3mfldversion}, $\tilde{q}^\ast: H^2(G)\to H^2(M)$ is an isomorphism.

We first show that the cup-product pairing $H^2(G)\otimes H^2(G)\to H^4(G)$ is non-degenerate.  Given a non-trivial element $\omega_1 \in H^2(G)$, $\tilde{q}^\ast(\omega_1)$ gives a non-trivial element of $H^2(M)_{\textrm{Tors}}$ since $\tilde{q}^\ast$ is injective.  By Lemma \ref{order}, there exists an element $\eta\in H^2(M)$ such that the order of $\lambda(\tilde{q}^\ast(\omega_1),\eta)$ is equal to the order of $\omega_1$, and since $\tilde{q}^\ast: H^2(G)\to H^2(M)$ is surjective, there exists an element $\omega_2\in H^2(G)$ such that $\tilde{q}^\ast(\omega_2)=\eta$.  Applying Lemma \ref{fundid} together with the fact that $\tilde{q}_\ast: H_0(M,\mathbb{Q}/\mathbb{Z})\to H_0(G,\mathbb{Q}/\mathbb{Z})$ is an isomorphism, we have
\[
0\ne \tilde{q}_\ast(\lambda(\tilde{q}^\ast(\omega_1),\eta))= \tilde{q}_\ast(\lambda(\tilde{q}^\ast(\omega_1),\tilde{q}^\ast(\omega_1)))= \tilde{q}_\ast([M])\frown \beta^{-1}(\omega_1\smile \omega_2).
\]
This shows that $\omega_1\smile \omega_2\ne 0$, and since $\omega_1$ was an arbitrary non-trivial element of $H^2(G)$ the cup product pairing $H^2(G) \otimes H^2(G)\to H^4(G)$ is non-degenerate.  

Note that the order of $\tilde{q}_\ast([M])\frown \beta^{-1}(\omega_1\smile \omega_2)$ divides the order of $\tilde{q}_\ast([M])$, so since 
\[
\ord(\omega_1)=\ord(\lambda(\tilde{q}^\ast(\omega_1),\eta))= \ord\left(\tilde{q}_\ast([M])\frown \beta^{-1}(\omega_1\smile \omega_2)\right),
\]
the order of $\omega_1$ divides $\ord(\tilde{q}_\ast([M]))$ for all $\omega_1\in H^2(G)$.  This shows that \linebreak $\ord(\tilde{q}_\ast([M]))\cdot H^2(G)=0$.  Since $H^2(G)\cong H^2(M)\cong H_1(M)$, $\ord(\tilde{q}_\ast([M]))$ annihilates $H_1(M)$ as well. 

\end{proof}

\section{Cup products in $H^\ast(G)$ for finite quotients $G\cong \pi_1(M)/N$ with $N<\pi_1(M)^{(2)}$}\label{cyclicity}

In this section we prove Lemma \ref{cupproducts} from the introduction.  Throughout this section, we will let $G$ be a finite group, and we will let $\rho: \pi_1(M)\cong \Gamma \to G$ a surjective homomorphism such that $\ker(\rho)\subseteq \Gamma^{(2)}$.  We will also let $Q= H_1(M)$, $N=[\Gamma,\Gamma]$, $K=[G,G]$, and we will refer to the maps labeled in the following commutative diagram of exact sequences:
 \[
\begin{array}[c]{ccccccccc}
1\to N \longrightarrow ~\Gamma \stackrel{q\circ \rho}{\longrightarrow} Q \to 1\\
~~~\downarrow \mbox{\scriptsize $r$}~~~~~~\downarrow  \mbox{\scriptsize $\rho$}~~~~\downarrow\mbox{\scriptsize $id$ } \\
1  \to  K \longrightarrow~ G \stackrel{q}{\longrightarrow} Q \to 1
\end{array}
\]
The above commutative diagram corresponds to a commutative diagram of continuous maps of the following form:
 \[
\begin{array}[c]{ccccccccc}
\widetilde{M}~ \longrightarrow M \stackrel{\tilde{q}\circ \tilde{\rho}}{\longrightarrow} BQ\\
~~~\downarrow \mbox{\scriptsize $\tilde{r}$}~~~~~~ \downarrow  \mbox{\scriptsize $\tilde{\rho}$}~~~~~\downarrow\mbox{\scriptsize $id$ }~~ \\
BK~ \rightarrow BG \stackrel{\tilde{q}}{\rightarrow} BQ,
\end{array}
\]
where $\widetilde{M}$ is the regular covering space of $M$ corresponding to $N<\pi_1(M)$.  The map $\tilde{\rho}^\ast$ induces a morphism between the Lyndon-Hochschild-Serre spectral sequence for the extension $1\to K\to G\to Q\to 1$ and the Cartan-Serre spectral sequence for the regular cover $\widetilde{M}\to M$.  The proof of Lemma 5 proceeds by analyzing this morphism.  We will require several preliminary lemmas.

\begin{lemma}\label{depth}  The cover $\widetilde{M}$ is a rational homology 3-sphere and $\tilde{r}^\ast: H^2(K)\to H^2(\widetilde{M})$ is an isomorphism.  
\end{lemma}

\begin{proof}  Since $\Gamma$ surjects onto $G$, $N\cong \pi_1(\widetilde{M})$ surjects onto $K$ and $H_1(\widetilde{M})$ surjects onto $H_1(K)$.  By assumption $\ker(\rho)\subseteq \Gamma^{(2)}\cong [N,N]= N^{(1)}$ and $r=\rho|_N$, so $\ker(r)\subseteq N^{(1)}$.  The abelianization map $N\to N/N^{(1)}\cong H_1(\widetilde{M})$ therefore factors through $r$, and so $\tilde{r}_\ast: H_1(\widetilde{M})\to H_1(K)$ is an isomorphism.  Since $K$ is a finite group, its abelianization $H_1(K)$ is also finite, so $H_1(M,\mathbb{Q})\cong H_1(K)\otimes \mathbb{Q}$ is trivial and therefore $\widetilde{M}$ is a rational homology 3-sphere.  The result then follows by Lemma \ref{3mfldversion} in section \ref{lpandcups}.
\end{proof}

To set some notation for the next lemm, let $(E_k)^\alpha$, $(d_k)^\alpha$ and $(E_k)^\beta$, $(d_k)^\beta$ denote the pages and differentials in the cohomological spectral sequences $H^r(Q,H^s(K))\implies H^{r+s}(G)$ and $H^r(Q,H^s(\widetilde{M}))\implies H^{r+s}(M)$ respectively.

\begin{lemma}\label{diag}  There exists a commutative diagram with rows of the following form:
\begin{equation}\label{CD0}
\begin{array}[c]{ccccccccc}
0\to (E_4^{1,2})^\alpha \longrightarrow ( E_3^{1,2})^\alpha \stackrel{(d_3^{1,2})^\alpha}{\longrightarrow} (E_3^{4,0})^\alpha \to (E_4^{4,0})^\alpha \to 0\\
~~~~~~~~~~~~~~~~\downarrow\mbox{\scriptsize $\rho_4^\ast$ }~~~~~~~~~~\downarrow  \mbox{\scriptsize $\rho_3^\ast$}~~~~~~~~~~~~~\downarrow\mbox{\scriptsize $\rho_3^\ast$}~~~~~~~~~\downarrow\mbox{\scriptsize $\rho_4^\ast$}~~~~~~~~~~~~~~ \\
0\to (E_4^{1,2})^\beta \longrightarrow ( E_3^{1,2})^\beta \stackrel{(d_3^{1,2})^\beta}{\longrightarrow} (E_3^{4,0})^\beta \to (E_4^{4,0})^\beta \to 0\\
\end{array}
\end{equation}
where the middle two homomorphisms labeled $\rho_3^\ast$ are isomorphisms.
\end{lemma}

\begin{proof}   Note that since the first cohomology group with integral coefficients is trivial for any finite group, $(E_2)^\alpha$ has the following form:

\begin{sseq}[grid=none, labels=none, 
entrysize=.79 cm ]{0...19}{0...3}
\ssmoveto 0 0
\ssdrop{\mbox{\small{{$\mathbb{Z}$}}}}
\ssmove 4 0
\ssdrop{\mbox{\small{0}}}
\ssmove 4 0
\ssdrop{\mbox{\small{$H^2(Q)$}}}
\ssmove 4 0
\ssdrop{\mbox{\small{$H^3(Q)$}}}
\ssmove 4 0
\ssdrop{\mbox{\small{$H^4(Q)$}}.}
\ssmoveto 0 1
\ssdrop{\mbox{\small{0}}}
\ssmove 4 0
\ssdrop{\mbox{\small{0}}}
\ssmove 4 0
\ssdrop{\mbox{\small{0}}}
\ssmove 4 0
\ssdrop{\mbox{\small{0}}}
\ssmove 4 0
\ssdrop{\mbox{\small{0}}}
\ssmoveto 0 2
\ssdrop{\mbox{\small{$H^0(Q,H^2(K))$}}}
\ssmove 4 0
\ssdrop{\mbox{\small{$H^1(Q,H^2(K))$}}}
\ssmove 4 0
\ssdrop{\mbox{\small{$H^2(Q,H^2(K))$}}}
\ssmove 4 0
\ssdrop{\mbox{\small{$H^3(Q,H^2(K))$}}}
\ssmove 4 0
\ssdrop{\mbox{\small{$H^4(Q,H^2(K))$}}}
\ssmoveto 0 3
\ssdrop{\mbox{\small{$H^0(Q,H^3(K))$}}}
\ssmove 4 0
\ssdrop{\mbox{\small{$H^1(Q,H^3(K))$}}}
\ssmove 4 0
\ssdrop{\mbox{\small{$H^2(Q,H^3(K))$}}}
\ssmove 4 0
\ssdrop{\mbox{\small{$H^3(Q,H^3(K))$}}}
\ssmove 4 0
\ssdrop{\mbox{\small{$H^4(Q,H^3(K))$}}}
\end{sseq}
By assumption the manifold $M$ is orientable, so $Q$ acts on $\widetilde{M}$ by orientation preserving homeomorphims.  It follows that $Q$ acts on $H^3(\widetilde{M},\mathbb{Z})\cong \mathbb{Z}$ trivially, and hence \linebreak $H^i(Q,H^3(\widetilde{M},\mathbb{Z}))\cong H^i(Q)$ for all $i$.  Furthermore, since $\widetilde{M}$ is a rational homology 3-sphere and $H^1(\widetilde{M})\cong \Hom(H_1(\widetilde{M}), \mathbb{Z})=0$, $(E_2)^\beta$ has the following form: 

\begin{sseq}[grid=none, labels=none, 
entrysize=.79 cm ]{0...19}{0...3}
\ssmoveto 0 0
\ssdrop{\mbox{\small{{$\mathbb{Z}$}}}}
\ssmove 4 0
\ssdrop{\mbox{\small{0}}}
\ssmove 4 0
\ssdrop{\mbox{\small{$H^2(Q)$}}}
\ssmove 4 0
\ssdrop{\mbox{\small{$H^3(Q)$}}}
\ssmove 4 0
\ssdrop{\mbox{\small{$H^4(Q)$}},}
\ssmoveto 0 1
\ssdrop{\mbox{\small{0}}}
\ssmove 4 0
\ssdrop{\mbox{\small{0}}}
\ssmove 4 0
\ssdrop{\mbox{\small{0}}}
\ssmove 4 0
\ssdrop{\mbox{\small{0}}}
\ssmove 4 0
\ssdrop{\mbox{\small{0}}}
\ssmoveto 0 2
\ssdrop{\mbox{\small{$H^0(Q,H^2(\widetilde{M}))$}}}
\ssmove 4 0
\ssdrop{\mbox{\small{$H^1(Q,H^2(\widetilde{M}))$}}}
\ssmove 4 0
\ssdrop{\mbox{\small{$H^2(Q,H^2(\widetilde{M}))$}}}
\ssmove 4 0
\ssdrop{\mbox{\small{$H^3(Q,H^2(\widetilde{M}))$}}}
\ssmove 4 0
\ssdrop{\mbox{\small{$H^4(Q,H^2(\widetilde{M}))$}}}
\ssmoveto 0 3
\ssdrop{\mbox{\small{{$\mathbb{Z}$}}}}
\ssmove 4 0
\ssdrop{\mbox{\small{$0$}}}
\ssmove 4 0
\ssdrop{\mbox{\small{$H^2(Q)$}}}
\ssmove 4 0
\ssdrop{\mbox{\small{$H^3(Q)$}}}
\ssmove 4 0
\ssdrop{\mbox{\small{$H^4(Q)$}}}
\end{sseq}
The map $\tilde{\rho}$ induces a morphism between these two spectral sequences, i.e. a sequence of homomorphisms $\rho^\ast_k: (E_k)^\alpha \to (E_k)^\beta$ such that $ \rho^\ast_k\circ(d_k)^\alpha =(d_k)^\beta\circ \rho^\ast_k $, $\rho_{k+1}^\ast$ is the map induced on homology by $\rho_{k}^\ast$, and the map $\rho^\ast_2: H^r(Q,H^s(K))\to H^r(Q,H^s(\widetilde{M}))$ is induced by the $Q$-module homomorphism $q^\ast:H^s(K)\to H^s(\widetilde{M})$.  Since $q^\ast: H^2(K)\to H^2(\widetilde{M})$ is an isomorphism by Lemma \ref{depth} and $q^\ast:  H^s(K)\to H^s(\widetilde{M})$ is trivially an isomorphism for $s\in \{0,1\}$, the maps $\rho^\ast_2:(E_2^{r,s})^\alpha\to (E_2^{r,s})^\beta$ are isomorphisms for all pairs $(r,s)$ such that $s\le 2$.

Since the first row of each spectral sequence vanishes, the differentials $(d_2^{i,2})^\alpha$, $(d_2^{i,2})^\beta$ and $(d_2^{i,1})^\alpha$, $(d_2^{i,1})^\beta$, whose domain or range lie in the first row, are trivial for all $i$.  It follows that $0$-th row of the $E_3$-page of each spectral sequences is identical to the $0$-th row of the $E_2$-page, i.e. $(E_3^{i,0})^\alpha\cong (E_2^{i,0})^\alpha$ and $(E_3^{i,0})^\beta \cong (E_2^{i,0})^\beta$ for all $i$.   Since none of the $d_2$-differentials have image lying in the $0$-th or first columns, it also follows that for $j\in \{0,1\}$, $(E_3^{j,2})^\alpha\cong (E_2^{j,2})^\alpha$ and $(E_3^{j,2})^\beta\cong (E_2^{j,2})^\beta$.  

Note that the $E_3^{1,2}$ terms of both spectral sequences are outside the range of the $d_3$ differential, and that the $d_3$ differential vanishes on the $E_3^{4,0}$ terms.  This implies that the rows in the following commutative diagram are exact:
\[
\begin{array}[c]{ccccccccc}
0\to (E_4^{1,2})^\alpha \longrightarrow ( E_3^{1,2})^\alpha \stackrel{(d_3^{1,2})^\alpha}{\longrightarrow} (E_3^{4,0})^\alpha \to (E_4^{4,0})^\alpha \to 0\\
~~~~~~~~~~~~~~~~\downarrow\mbox{\scriptsize $\rho_4^\ast$ }~~~~~~~~~~\downarrow  \mbox{\scriptsize $\rho_3^\ast$}~~~~~~~~~~~~~\downarrow\mbox{\scriptsize $\rho_3^\ast$}~~~~~~~~~\downarrow\mbox{\scriptsize $\rho_4^\ast$}~~~~~~~~~~~~~~ \\
0\to (E_4^{1,2})^\beta \longrightarrow ( E_3^{1,2})^\beta \stackrel{(d_3^{1,2})^\beta}{\longrightarrow} (E_3^{4,0})^\beta \to (E_4^{4,0})^\beta \to 0\\
\end{array}
\]
The fact that $\rho^\ast_2$ induces isomorphisms on $E_2^{r,s}$ for $s\le2$ implies that $\rho^\ast_3$ induces isomorphisms $(E_3^{1,2})^\alpha \to (E_3^{1,2})^\beta$ and $(E_3^{4,0})^\alpha \to (E_3^{4,0})^\beta$, since $\rho^\ast_3$ is induced by $\rho^\ast_2$ and each of these groups are isomorphic to the corresponding entries on the $E_2$-page of the spectral sequence. This shows that the middle two homomorphisms in the above commutative diagram are isomorphisms.

\end{proof}

\begin{lemma}\label{BL}  The term $(E_4^{1,2})^\beta$ is trivial.
\end{lemma}

\begin{proof}
The term $(E_4^{1,2})^\beta$ is isomorphic to  $(E_\infty^{1,2})^\beta$, so it suffices to show that $(E_\infty^{1,2})^\beta$ is trivial.  The term $(E_\infty^{1,2})^\beta$ lies on the third diagonal of the $E_\infty$-page for the spectral sequence $H^{r}(Q,H^s(\widetilde{M}))\implies H^{r+s}(M)$.  Since $H^{r}(Q,H^s(\widetilde{M}))$ is annihilated by $|Q|$ for any $r\ge 1$, the group $E_k^{r,s}$ is torsion for all $r\ge1$.  The groups $(E_\infty^{i,3-i})^\beta$ for give successive quotients in the filtration 
\[
(E_\infty^{3,0})^\beta\cong F^3_3\subseteq F^3_2\subseteq F^3_1\subseteq F^3_0= H^3(M)\cong \mathbb{Z}.
\] 
Since $\mathbb{Z}$ is torsion-free and $(E_\infty^{3,0})^\beta$ is torsion, $(E_\infty^{3,0})^\beta\cong 0$.  This implies that $(E_\infty^{2,1})^\beta\cong F^3_2/F^3_3\cong F^3_2/(E_\infty^{3,0})\cong F^3_2$, so $F^3_2\subseteq \mathbb{Z}$ is torsion and hence trivial as well.  Applying this argument once more, we find that $(E_\infty^{1,2})^\beta\cong F^3_1/F^3_2$ must vanish as well.  
\end{proof}

\begin{lemma}\label{E4isom}  The term $(E_4^{4,0})^\alpha$ is isomorphic to $(E_4^{4,0})^\beta$.
\end{lemma}

\begin{proof}

By Lemma \ref{BL}, the term in bottom left of the commutative diagram (\ref{CD0}) from Lemma \ref{diag} is trivial.  The commutativity of the diagram together with the fact that the second vertical map is an isomorphism shows that $(E_4^{1,2})^\alpha$ is also trivial. Since the first $3$ maps in this commutative diagram of exact sequences are isomorphisms, a straightforward diagram chasing argument shows that the last map $\rho_4^\ast:(E_4^{4,0})^\alpha \to (E_4^{4,0})^\beta$ is an isomorphism as well.

\end{proof} 

\begin{lemma}\label{E4cyclicity} The term $(E_4^{4,0})^\beta$ is cyclic.
\end{lemma}

\begin{proof}

Note that there is an exact sequence
\[
(E_4^{0,3})^\beta \stackrel{(d_4^{4,0})^\beta}{\longrightarrow} (E_4^{4,0})^\beta \to (E_5^{4,0})^\beta \to 0.
\]

Since $E_5^{4,0}\cong E_\infty^{4,0}$ and $H^4(\Gamma)=0$, it follows that $(d_4^{4,0})^\beta$ is surjective.  The group $(E_4^{0,3})^\beta$ is isomorphic to a subgroup of $(E_2^{0,3})^\beta\cong H^0(Q,H^3(\widetilde{M}))\cong \mathbb{Z}$, so since $(E_4^{4,0})^\beta$ is isomorphic to a quotient of  $(E_4^{0,3})^\beta$,  $(E_4^{4,0})^\beta$ is cyclic. 
\end{proof}

\begin{lemma}\label{cyclicimage}  The map $q^\ast:H^4(Q)\to H^4(G)$ has cyclic image.
\end{lemma}

\begin{proof}
Recall that the image of $q^\ast:H^4(Q)\to H^4(G)$ is isomorphic to $(E_\infty^{4,0})^\alpha$, and that $(E_\infty^{4,0})^\alpha$ is isomorphic to a quotient of $(E_4^{4,0})^\alpha$.  By Lemma \ref{E4isom}, $(E_4^{4,0})^\alpha\cong (E_4^{4,0})^\beta$ and by Lemma \ref{E4cyclicity} $(E_4^{4,0})^\beta$ is cyclic.  Since quotients of cyclic groups are cyclic, the result follows.
\end{proof}

We are now ready to prove Lemma \ref{cupproducts}, which is an immediate consequence of Lemma \ref{H2surj} from the previous section and Lemma \ref{cyclicimage}.

\begin{proof}[Proof of Lemma \ref{cupproducts}]  Let $\omega_1$ and $\omega_2$ be elements of $H^2(G)$.  By Lemma \ref{H2surj} , $q^\ast:H^2(Q)\to H^2(G)$ is surjective, so there exist $\alpha_1,~\alpha_2\in H^2(Q)$ such that $q^\ast(\alpha_1)=\omega_1$ and $q^\ast(\alpha_2)=\omega_2$.  It follows that
\[
\omega_1\smile \omega_2=q^\ast(\alpha_1)\smile q^\ast(\alpha_2)=q^\ast(\alpha_1\smile\alpha_2),
\]
so any cup product of elements in $H^2(G)$ lies in $q^\ast(H^4(Q))$.  By Lemma \ref{cyclicimage}, the image of $q^\ast:H^4(Q)\to H^4(G)$ is cyclic. 
\end{proof}
\section{The proof of Theorem \ref{mainthm}}

We now turn to the proof of Theorem \ref{mainthm}.
\begin{proof}[Proof of Theorem \ref{mainthm}]  Since $\ker(q)\subseteq \Gamma^{(2)}$ and $\Gamma^{(2)}\subseteq \Gamma^{(1)}$, the non-degeneracy of the cup product pairing follows from Theorem \ref{features}.  The cyclicity of the image $C$ of the the cup product pairing $H^2(G)\otimes H^2(G)\to H^4(G)$ follows from Lemma \ref{cyclicimage}, so it remains to demonstrate the existence of the desired embedding $i:C\to \mathbb{Q}/\mathbb{Z}$.  

Let $\psi:H^4(G)\to H_0(G,\mathbb{Q}/\mathbb{Z})$ denote the map $\alpha\mapsto \tilde{q}_\ast([M])\frown \beta^{-1}(\alpha)$, and let $i$ denote the restriction of $\psi$ to $C$.  By Lemma 2, given $\omega_1$ and $\omega_2$ in $H^2(G)$, 
\begin{equation}\label{lambdaid}
i(\omega_1\smile \omega_2)=\tilde{q}_\ast([M])\frown\beta^{-1}(\omega_1\smile \omega_2)=\tilde{q}_\ast\left(\lambda( \tilde{q}^\ast(\omega_1), \tilde{q}^\ast(\omega_2))\right).
\end{equation}
After identifying $H_0(G,\mathbb{Q}/\mathbb{Z})$ and $H_0(M,\mathbb{Q}/\mathbb{Z})$ with $\mathbb{Q}/\mathbb{Z}$ in the natural way, the last term of this equation is equal to $\langle[M]\frown \tilde{q}^\ast(\omega_1),[M]\frown \tilde{q}^\ast(\omega_2)  \rangle$.

It remains to check that $i:C\to \mathbb{Q}/\mathbb{Z}$ is injective.  Given an abelian group $A$, let $\exp(A)$ denote the maximal order of an element of $A$.  Note that for a finite abelian group $\exp(A)=\exp(A\otimes A)$.  Since $C$ is cyclic and is isomorphic to a quotient of $H^2(G)\otimes H^2(G)$, it follows that the order of $C$ divides $\exp(H^2(G)\otimes H^2(G))=\exp(H^2(G))$.   It therefore suffices to show that $i(C)$ contains an element of order $\exp(H^2(G))$.

By Lemma \ref{order} there exist elements $\eta_1, \eta_2 \in H^2(M)$ such that the order of $\lambda(\eta_1,\eta_2)$ is equal to $\exp(H^2(M))$.  Since $\ker(q)\subseteq \Gamma^{(1)}$, $q^\ast:H^2(G)\to H^2(M)$ is an isomorphism by Lemma 3, so $\exp(H^2(G))=\exp(H^2(M))$, and there exist elements $\omega_1$ and $\omega_2$ such that $q^\ast(\omega_i)=\eta_i$.  Equation (\ref{lambdaid}) above therefore shows that 
\[
i(\omega_1\smile \omega_2)=\tilde{q}_\ast(\lambda(\eta_1,\eta_2)),
\]
and since $\tilde{q}_\ast:H_0(M,\mathbb{Q}/\mathbb{Z})\to H_0(G,\mathbb{Q}/\mathbb{Z})$ is an isomorphism, 
\[
\ord(i(\omega_1\smile \omega_2))=\ord(\tilde{q}_\ast(\lambda(\eta_1,\eta_2)))=\ord(\lambda(\eta_1,\eta_2))=\exp(H^2(G)).
\]
\end{proof}

\end{document}